\documentclass{article}
\DefineBibliographyExtras{english}{%
  }
\AtBeginBibliography{\sloppy}
\defbibheading{nomark}[\refname]{\section*{#1}}
\defbibheading{nomarksub}[\refname]{\subsection*{#1}}

\makeatletter
\patchcmd{\abstract}{\leftmargin3pc}{\leftmargin2.5pc}{}{\typeout{*** abstract patch failed ***}}
\makeatother

\usepackage{booktabs}     % For professional table rules (\toprule, \midrule, \bottomrule)
\usepackage{bigdelim} 
\usepackage{multirow}
\usepackage{needspace}
\usepackage{arydshln}   % must be loaded after booktabs/array
\usepackage[normalem]{ulem}
\usepackage{float}
\newcommand{\PP}{\mathbb{P}}

\DeclareMathAlphabet{\mathcal}{OMS}{cmsy}{m}{n}
\DeclareMathOperator{\Image}{Im}
\DeclareMathOperator{\SL}{SL}
\DeclareMathOperator{\GL}{GL}

\newcommand{\pCs}{C_s(p)}%{\texttt{pCs}}
\newcommand{\pNs}{C_{ns}(p)}%{\texttt{pNs}}
\newcommand{\normSplitCartan}[1]{C_{s}^+(#1)}%{\texttt{#1Cn}}
\newcommand{\normNonSplit}[1]{C_{ns}^+(#1)}

\renewcommand{\le}{\leqslant}
\renewcommand{\leq}{\leqslant}
\renewcommand{\ge}{\geqslant}
\renewcommand{\geq}{\geqslant}

\newcommand{\Fp}{\mathbb{F}_p}%\mathbb{Z}/p\mathbb{Z}}
\newcommand{\cupw}{\, \cup \;\,}

\title[Completing the classification of torsion over sextic fields for elliptic curves defined over $\mathbb{Q}$]{COMPLETING THE CLASSIFICATION OF TORSION OVER SEXTIC FIELDS FOR ELLIPTIC CURVES DEFINED OVER $\mathbb{Q}$}
\begin{document}

\setlength{\abovedisplayskip}{3pt}
\setlength{\belowdisplayskip}{3pt}
\makeatletter
\newtheoremstyle{tight}
  {5pt}   % space above
  {5pt}   % space below
  {\itshape} % body font
  {}      % indent
  {\bfseries} % head font
  {.}     % punctuation after theorem name
  {0.5em} % space after theorem name
  {}      % theorem head spec
\makeatother

\makeatletter
\renewcommand{\section}{\@startsection{section}{1}%
  %\z@{2.5ex plus 1ex minus .2ex}{1.5ex plus .5ex}%
  \z@{2.3ex plus 1ex minus .2ex}{1ex}%
  {\normalfont\bfseries}}
\renewcommand{\subsection}{\@startsection{subsection}{2}%
  \z@{2.0ex plus .9ex minus .2ex}{.8ex plus .2ex}%
  {\normalfont\bfseries}}
\makeatother

\theoremstyle{tight}

\makeatletter
\def\thm@space@setup{\thm@preskip=0pt
\thm@postskip=0pt}
\makeatother

%
%%%%%%%%%%%%%%%%%%%%
\newcounter{ctr}
\newtheorem{theorem}{Theorem}[section]
\newtheorem{proposition}[theorem]{Proposition}
\newtheorem{lemma}[theorem]{Lemma}
\newtheorem{corollary}[theorem]{Corollary}
\newtheorem{conjecture}[theorem]{Conjecture}
\newtheorem*{acknowledgments}{Acknowledgments}
\theoremstyle{definition}
\newtheorem*{definition}{Definition}
\theoremstyle{remark}
\newtheorem*{notation}{Notation}
\newtheorem{remark}[theorem]{Remark}
\newtheorem*{question}{Question}
\newtheorem{conj}[theorem]{Conjecture}

\author{Nikola Adžaga, Tomislav Gužvić}
\makeatletter

\patchcmd{\@setauthors}{\scshape}{\relax}{}{}
\patchcmd{\@author}{\MakeUppercase}{\relax}{}{}
\renewcommand{\author}[1]{\def\@author{#1}}

\makeatother
\address{Department of Mathematics, University of Zagreb Faculty of Civil Engineering}
\email{nikola.adzaga@grad.unizg.hr}
\address{{\vskip -2.5em}Department of Mathematics, University of Zagreb Faculty of Civil Engineering}
\address{{\vskip -3em} Ericsson Nikola Tesla, Croatia}
\email{tguzvic@gmail.com}

\begin{abstract}
We complete the classification of torsion subgroups $E(K)_\tors$ occurring for an elliptic curve $E/\mathbb{Q}$ over a sextic number field $K$. Previous work determined the complete set of possible groups up to one remaining case, namely $C_3 \oplus C_{18}$. We prove that this group does not occur.
\end{abstract}
%%%%%%%%%%%%%%%%%%%%%%%%%%%%%%%%%%%%%%%%%%%
\maketitle
%%%%%%%%%%%%%%%%%%%%%%%%%%%%%%%%%%%%%%%%%%%
\vskip -0.8em \section{Introduction}
Let $K$ be a number field, and $E/K$ be an elliptic curve. A famous theorem of Mordell and Weil states that $E(K)$ is a finitely generated abelian group. Therefore, we have that $E(K) \cong E(K)_{tors} \oplus \mathbb{Z}^{r}$, for some integer $r \ge 0$.

A result of Mazur classifies the possible torsion structures of elliptic curves defined over the rational numbers. More precisely, we have:
\setlength{\topsep}{2em}

\begin{theorem}[\cite{mazurtorzija}]
    Let $E/\mathbb{Q}$ be an elliptic curve. Then 
    $E(\mathbb{Q})_{tors}$ is isomorphic to one of the following groups: \vspace*{-0.45em}
    \[C_{n}, \; 1 \le n \le 10, 12,\] \vspace*{-1.45em}
    \[ C_{2} \oplus C_{2n}, \; n \le 4.\]
\end{theorem}

One may ask what happens over the number fields of fixed degree $d$. For quadratic fields ($d=2$), the classification includes $11$ additional possible groups, as shown by Kenku, Momose \cite{kenkumomose}, and Kamienny \cite{kamienny}. In the results mentioned so far, all the possible torsion structures (over $\Q$ and over quadratic fields) appear \emph{infinitely often}, i.~e.~for infinitely many non-isomorphic curves. For cubic fields, a new phenomenon appears: Jeon, Kim and Schweizer determined all torsion structures that appear infinitely often \cite{JKS}, while Najman found that there is a unique elliptic curve over a cubic field with $\Z/21\Z$-torsion \cite{najman_ratcubic}. This was later completed by Derickx, Etropolski, van Hoeij, Morrow, and Zureick-Brown, who proved that the groups found by Jeon--Kim--Schweizer together with Najman's sporadic example account for all possible torsion subgroups over cubic fields \cite{DEHMZB2021}.

Derickx and Najman completed the classification of the possible torsion groups of elliptic curves over quartic fields (all of which occur infinitely often) \cite{DerickxNajman2025}, and Najman recently completed the analogous classification over quintic fields \cite{Najman_quintic}.

The classification of torsion subgroups of elliptic curves that appear infinitely often over number fields of fixed degree has been previously completed for quartic fields by Jeon, Kim, and Park \cite{JeonKimPark2006}, and for quintic and sextic fields by Derickx and Sutherland \cite{DerickxSutherland2017}. More recently, the cases $d=7,8,9$ have been studied by Najman and Varivoda \cite{NV}.

A natural variant of this problem is the following: instead of allowing elliptic curves defined over arbitrary number fields of degree $d$, one restricts to elliptic curves defined over $\mathbb{Q}$, and asks how their torsion subgroups can behave over number fields of fixed degree $d$.

For a positive integer $d$, let $\Phi_{\mathbb{Q}}(d)$ denote the set of isomorphism classes of groups $E(K)_{tors}$, where $E$ ranges over all elliptic curves defined over $\mathbb{Q}$ and $K$ ranges over all number fields of degree $d$ over $\mathbb{Q}$. In other words, $\Phi_{\mathbb{Q}}(d)$ records all possible torsion subgroups of $E/\Q$ that can arise over degree-$d$ extensions of $\mathbb{Q}$.

The sets $\Phi_{\mathbb{Q}}(d)$ have been determined for several values of $d$.
In the case when $d=2$ or $3$, the sets $\Phi_{\mathbb{Q}}(d)$ have been completely described by Najman \cite{najman_ratcubic}:
\begin{align*}
\Phi_{\mathbb{Q}}(2) &= \{C_n \mid n=1,\dots,10,12,15,16\} \cup \\ & \quad \, \cup \{C_2 \oplus C_{2m} \mid m=1,\dots,6\} \cup  \{C_3 \oplus C_3, \, C_3 \oplus C_6, \, C_4 \oplus C_4\}, \nonumber\\
\Phi_{\mathbb{Q}}(3) &= \{C_n \mid n=1,\dots,10,12,13,14,18,21\} \cup \{C_2 \oplus C_{2m} \mid m=1,2,3,4,7\}.
\end{align*}
The case $d=4$ has been settled by Chou \cite{chou2} and González Jiménez and Najman \cite{growth}. For $d=5$, this classification was obtained by González Jiménez \cite{Enrique} and more generally, for $d=p$ with $p\geqslant 7$ prime, by González Jiménez and Najman \cite{growth}. Moreover, it was proved in \cite{growth} that $\Phi_{\mathbb{Q}}(d)=\Phi_{\mathbb{Q}}(1)$ if the smallest prime divisor of $d$ is greater than $7$.

More recently, Gu\v{z}vi\'c \cite{guzvic_pq_2023} determined the sets $\Phi_{\mathbb{Q}}(pq)$, where $p,q$ are primes with $pq \notin \{4,6\}$. For $d=6$, the classification was started in \cite{DGJ20}, while a nearly complete result was obtained in \cite{tomi1}. More precisely, a set $S$ was determined such that at most one group $G \in S$ lies outside $\Phi_{\mathbb{Q}}(6)$, namely $G \cong C_3 \oplus C_{18}$. We prove that $C_3 \oplus C_{18}$ is not in $\Phi_{\Q}(6)$.

Our proof relies on the theory of Galois representations attached to elliptic curves. The assumed existence of a $C_3 \oplus C_{18}$ torsion subgroup would impose strong simultaneous constraints on the mod-$2$ and $3$-adic Galois representations of the curve. By applying the recent classification of $\ell$-adic Galois images for elliptic curves over $\Q$, we translate these arithmetic constraints into a problem in Diophantine geometry: the $j$-invariant of such a curve must correspond to a rational point on one of finitely many modular curves. We then analyze these curves using classical methods and show that none has the required rational points corresponding to elliptic curves without complex multiplication, thereby proving our main result, which had previously been conjectured by Daniels and Gonz\'alez-Jim\'enez \cite{DGJ20}.
\enlargethispage{\baselineskip}
\begin{theorem}\label{tm:complete}%\vskip -0.3em
If $K$ is a sextic number field and $E/\Q$ is an elliptic curve, then 
$E(K)_{\tors}$ is one of the following groups:
\begin{itemize}[itemsep=1pt, topsep=1pt]
    \item $C_m$, $m=1,\dots,16,18,21,30$, $m \neq 11$, \nopagebreak
    \item $C_2 \oplus C_{2m}$, $m=1,\dots,7,9$,\nopagebreak
    \item $C_3 \oplus C_{3m}$, $m=1,2,3,4$,\nopagebreak
    \item $C_4 \oplus C_{4m}$, $m=1,3$,\nopagebreak
    \item $C_6 \oplus C_6$.
\end{itemize}
\end{theorem}
\noindent All of the groups listed in Theorem \ref{tm:complete} indeed appear as $E(K)_{tors}$ for some elliptic curve $E/\Q$ and some sextic field $K$. In particular,
\[
\Phi_{\mathbb{Q}}(6) = \Phi_{\mathbb{Q}}(2) \cup \Phi_{\mathbb{Q}}(3) \cup \{C_{30}, \, C_2 \oplus C_{18}, \, C_3 \oplus C_9, \, C_3 \oplus C_{12}, \, C_4 \oplus C_{12}, \, C_6 \oplus C_6\}.
\]

The classification of torsion structures of elliptic curves defined over $\mathbb{Q}$ after base change to number fields whose degree is the product of two distinct primes is now complete.

The code we developed to support and verify the computations in this paper, written in Magma \cite{magma}, is available at \[ \text{\url{https://github.com/NikolaAdzaga/SexticTorsion}.}\]

\section{Galois representations}\label{sec:Galois}

Let \(E/\mathbb{Q}\) be an elliptic curve. Recall that the group of \(n\)-torsion points of \(E(\Qbar)\) is \(E[n] \cong \mathbb{Z}/n\mathbb{Z} \oplus \mathbb{Z}/n\mathbb{Z}\).
The absolute Galois group \( G_{\Q} = \mathrm{Gal}(\overline{\mathbb{Q}}/\mathbb{Q})\)
naturally acts on these torsion points, while respecting the group structure: every automorphism of
\(\overline{\mathbb{Q}}\) acts on the coordinates of the points in \(E[n]\),
and therefore induces an automorphism of \(E[n]\). In this way we obtain
a homomorphism $G_{\Q} \longrightarrow \mathrm{Aut}(E[n]) \cong \mathrm{GL}_2(\mathbb{Z}/n\mathbb{Z})$,
which is called the \emph{mod \(n\) Galois representation} attached to \(E\),
and is usually denoted by \(\rho_{E,n}\) or just \(\rho_n\).% (if the elliptic curve \(E\) is clear from the context).

The coordinates of the $n$-torsion points \(E[n]\) are
algebraic numbers, and the smallest field containing these coordinates is the \emph{\(n\)-division field of \(E\)}, denoted \(\mathbb{Q}(E[n])\). Any automorphism
\(\sigma \in \mathrm{Gal}(\Qbar/\mathbb{Q}(E[n]))\) fixes all
points of \(E[n]\). Since such automorphisms act trivially on \(E[n]\), the action of
$G_{\Q}$ on
\(E[n]\) factors through the smaller Galois group
\(\mathrm{Gal}(\mathbb{Q}(E[n])/\mathbb{Q})\). Concretely, instead of tracking
how every automorphism of \(\overline{\mathbb{Q}}\) acts on torsion points, it
is enough to understand how the Galois group of the $n$-division field
\(\mathbb{Q}(E[n])\) acts. We will denote the image $\rho_{E,n}(G_{\Q})$ by $G_E(n)$.
Since $\Q(E[n])$ is a Galois extension of $\Q$ and
$\ker \rho_{E,n} = \Gal(\overline{\Q}/\Q(E[n]))$,
by the first isomorphism theorem we have $G_E(n) \cong \Gal(\Q(E[n])/\Q)$.

 Fixing a basis of \(E[n]\) identifies \(E[n]\) with \((\mathbb{Z}/n\mathbb{Z})^2\), so that points of order \(n\) correspond to vectors of additive order \(n\). For a point \(P\in E[n]\) with corresponding vector \(v\), the field of definition \(\mathbb{Q}(P)\) is the fixed field of the stabilizer of \(v\) in \(G_E(n)\). By Galois theory and the orbit--stabilizer theorem,
\(
[\Q(P):\Q]
=[G_E(n):\operatorname{Stab}_{G_E(n)}(v)]
=|G_E(n)\cdot v|.
\)
In this way, the degrees of the fields of definition of points in \(E[n]\) can be determined from the action of \(G_E(n)\) on \((\Z/n\Z)^2\).

Isogenies admit a similar description. A cyclic subgroup \(C\subseteq E[n]\) of order \(n\) is the kernel of an \(n\)-isogeny defined over \(\mathbb{Q}\) if and only if \(C\) is stable under the action of \(G_{\mathbb{Q}}\). %, that is, if and only if the corresponding cyclic subgroup of \( (\Z/n\Z)^2\) is stabilized by \(G_E(n)\).
Thus \(E\) admits a rational \(n\)-isogeny if and only if \(G_E(n)\) stabilizes a cyclic subgroup of order \(n\) in \((\mathbb{Z}/n\mathbb{Z})^2\).

Now let $n=p$ be prime. For elliptic curves without complex multiplication (CM), surjectivity of $\rho_p$ is typical: Serre's open image theorem \cite{Serre1971} implies that for each elliptic curve without CM, there are only finitely many primes $p$ such that $\rho_p$ is not surjective. Such non-surjective images and the corresponding primes are called \emph{exceptional}. Whether the exceptional primes are uniformly bounded, independently of $E$, is a well-known open question of Serre. For more background on Galois representations attached to elliptic curves and on the subgroup structure of $\GL_2(\Fp)$, see \cite{Serre1971,SUTHERLAND_2016,zywina} and Appendix~\ref{app:background}.

For small primes, Zywina \cite{zywina} classified the possible exceptional images for elliptic curves $E$ without CM. Remarkably, for a given image $G_E(p)$ of $\rho_p$, the $j$-invariant of $E$ was parameterized: he provides an explicit \emph{$j$-map} $J_G\in\Q(t)$ with the property that $G_E(p)$ is conjugate to a subgroup of $\langle G,-I\rangle$ if and only if $j(E)=J_G(t)$ for some $t\in\Q$. For subgroups of $\GL_2(\Fp)$ we use Sutherland's labels; see \cite[Sec.~6.4, p.~61]{SUTHERLAND_2016} and Appendix \ref{app:labels} for a dictionary (including the LMFDB \cite{lmfdb} labels).

For $p=2$, the only image we will encounter is the Borel subgroup $\texttt{2B}$ of upper-triangular matrices in $\GL_2(\mathbb{F}_2)$. Indeed, it was proved in \cite[Theorem 18]{tomi1} that if $C_3\oplus C_{18}\subseteq E(K)_{\tors}$ for a number field $K$ of degree $6$, then $G_E(2)=\texttt{2B}$. In particular, Zywina's classification implies that there exists $t\in\Q^{\times}$ such that
\( j(E)=j_{\texttt{2B}}(t)=256(t+1)^3/t \).

For $p=3$, all possible images $G_E(3)$ for a non-CM elliptic curve $E/\Q$ are collected in Table~\ref{tab:mod3}, which is compiled from \cite[\S 1.2 and Theorem 1.2]{zywina} and \cite[Table~3, p.~64]{SUTHERLAND_2016}. For each exceptional image we list generators, the degree $[\Q(E[3]):\Q]=|G_E(3)|$, and the associated $j$-map; if $\rho_3$ is surjective, then $[\Q(E[3]):\Q]=48$.

\begin{table}[H]
\centering
{\small
\begin{tabular}{@{}lccc@{}}
Group & Generators & $|G_E(3)|$ & $j$-map \\
\addlinespace
\texttt{3Cs.1.1} & $\left(\begin{smallmatrix} 1 & 0 \\ 0 & 2 \end{smallmatrix}\right)$ & 2 & \,\, \multirow{2}{*}{$J_1(t)=\dfrac{27(t+1)^3(t+3)^3(t^2+3)^3}{t^3(t^2+3t+3)^3}$} \\
\texttt{3Cs} & $\left(\begin{smallmatrix} 2 & 0 \\ 0 & 2 \end{smallmatrix}\right), \left(\begin{smallmatrix} 1 & 0 \\ 0 & 2 \end{smallmatrix}\right)$ & 4 & \\
\midrule
\texttt{3B.1.1} & $\left(\begin{smallmatrix} 1 & 0 \\ 0 & 2 \end{smallmatrix}\right), \left(\begin{smallmatrix} 1 & 1 \\ 0 & 1 \end{smallmatrix}\right)$ & 6 & \multirow{3}{*}{$J_3(t) = \dfrac{27(t+1)(t+9)^3}{t^3}$} \\
\texttt{3B.1.2} & $\left(\begin{smallmatrix} 2 & 0 \\ 0 & 1 \end{smallmatrix}\right), \left(\begin{smallmatrix} 1 & 1 \\ 0 & 1 \end{smallmatrix}\right)$ & 6 & \\
\texttt{3B} & $\left(\begin{smallmatrix} 2 & 0 \\ 0 & 2 \end{smallmatrix}\right), \left(\begin{smallmatrix} 1 & 0 \\ 0 & 2 \end{smallmatrix}\right), \left(\begin{smallmatrix} 1 & 1 \\ 0 & 1 \end{smallmatrix}\right)$ & 12 & \\
\midrule
\texttt{3Ns} & $\left(\begin{smallmatrix} 2 & 0 \\ 0 & 2 \end{smallmatrix}\right), \left(\begin{smallmatrix} 0 & 2 \\ 1 & 0 \end{smallmatrix}\right), \left(\begin{smallmatrix} 1 & 0 \\ 0 & 2 \end{smallmatrix}\right)$ & 8 & $J_2(t)=\dfrac{ 27(t+1)^3(t-3)^3}{t^3}$ \\
\midrule
\texttt{3Nn} & $\left(\begin{smallmatrix} 1 & 0 \\ 0 & 2 \end{smallmatrix}\right), \left(\begin{smallmatrix} 2 & 1 \\ 2 & 2 \end{smallmatrix}\right)$ & 16 & $J_4(t) = t^3$ \\
\end{tabular}}
\caption{The exceptional images $G_E(3)$ for non-CM elliptic curves $E/\Q$.}\label{tab:mod3}
\end{table}

At one point, we will also need to work with images modulo $9$, and correspondingly, with subgroups of prime-power level. Sutherland and Zywina \cite{34} determined all open subgroups $G\leq \GL_2(\widehat{\Z})$ of prime-power level, containing $-I$ and with surjective determinant, for which the modular curve $X_G$ has infinitely many rational points. The groups of genus $0$ among them are listed in \cite[Table~1]{34}; each carries a Sutherland--Zywina (SZ) label, and complete data for these groups can be found in the online supplement \cite{SZsupp} to \cite{34}. Table~\ref{tab:SZ9} reproduces all such groups of $3$-power level, that is, the sixteen groups of level $9$ together with those of levels $1$, $3$ and $27$. The group of level $1$ is $\GL_2(\widehat{\Z})$ itself, with $X_G=X(1)$. For each group $G$ we give its SZ label, its label in the LMFDB (see also Appendix~\ref{app:labels}), and the map $F(t)$ from $X_G$ to the modular curve of the group in the last column; composing these maps along the last column, down to the curve $X(1)$, yields the $j$-map of $X_G$. For example,
\(j_{9B^{0}-9a}(t)=j_{3B^{0}-3a}\!\left(t(t^2+9t+27)\right)=\frac{\left(t(t^2+9t+27)+3\right)^3\left(t(t^2+9t+27)+27\right)}{t(t^2+9t+27)}=\frac{(t^3+9t^2+27t+3)^3(t+3)^3}{t(t^2+9t+27)}\).
\begin{table}[H]
\centering
{\small{
\begin{tabular}{@{}llll@{}}
SZ label & LMFDB label & $F(t)$ & covers \\
\addlinespace
$1A^{0}-1a$ & 1.1.0.a.1 & --- & --- \\
\addlinespace
$3A^{0}-3a$ & 3.3.0.a.1 & $t^3$ & $X(1)$ \\
$3B^{0}-3a$ & 3.4.0.a.1 & $(t+3)^3(t+27)/t$ & $X(1)$ \\
$3C^{0}-3a$ & 3.6.0.b.1 & $(t-9)(t+3)/t$ & $3A^{0}-3a$ \\
$3D^{0}-3a$ & 3.12.0.a.1 & $729/(t^3-27)$ & $3B^{0}-3a$ \\
\addlinespace
$9A^{0}-9a$ & 9.9.0.a.1 & $t^3+9t-6$ & $3A^{0}-3a$ \\
$9B^{0}-9a$ & 9.12.0.a.1 & $t(t^2+9t+27)$ & $3B^{0}-3a$ \\
$9C^{0}-9a$ & 9.12.0.b.1 & $t^3$ & $3B^{0}-3a$ \\
$9D^{0}-9a$ & 9.18.0.d.1 & $-27/t^3$ & $3C^{0}-3a$ \\
$9E^{0}-9a$ & 9.18.0.a.1 & $-9(t^3+3t^2-9t-3)/(8t^3)$ & $3C^{0}-3a$ \\
$9F^{0}-9a$ & 9.27.0.a.1 & $J_{9F}(t)$, displayed below & $X(1)$ \\
$9G^{0}-9a$ & 9.27.0.b.1 & $F_{9G}(t)$, displayed below & $3A^{0}-3a$ \\
$9H^{0}-9a$ & 9.36.0.b.1 & $3(t^3+9)/t^3$ & $3D^{0}-3a$ \\
$9H^{0}-9b$ & 9.36.0.a.1 & $3(t^3+9t^2-9t-9)/(t^3-9t^2-9t+9)$ & $3D^{0}-3a$ \\
$9H^{0}-9c$ & 9.36.0.c.1 & $-6(t^3-9t)/(t^3+9t^2-9t-9)$ & $3D^{0}-3a$ \\
$9I^{0}-9a$ & 9.36.0.d.2 & $-6(t^3-9t)/(t^3-3t^2-9t+3)$ & $9B^{0}-9a$ \\
$9I^{0}-9b$ & 9.36.0.e.1 & $-3(t^3+9t^2-9t-9)/(t^3+3t^2-9t-3)$ & $9B^{0}-9a$ \\
$9I^{0}-9c$ & 9.36.0.d.1 & $(t^3-6t^2+3t+1)/(t^2-t)$ & $9B^{0}-9a$ \\
$9J^{0}-9a$ & 9.36.0.f.2 & $(t^3-3t+1)/(t^2-t)$ & $9C^{0}-9a$ \\
$9J^{0}-9b$ & 9.36.0.f.1 & $-18(t^2-1)/(t^3-3t^2-9t+3)$ & $9C^{0}-9a$ \\
$9J^{0}-9c$ & 9.36.0.g.1 & $3(t^3+3t^2-9t-3)/(t^3-3t^2-9t+3)$ & $9C^{0}-9a$ \\
\addlinespace
$27A^{0}-27a$ & 27.36.0.a.1 & $t^3$ & $9B^{0}-9a$ \\ \hline
\end{tabular}}}{\small
\[
J_{9F}(t)=\frac{3^7(t^2-1)^3(t^6+3t^5+6t^4+t^3-3t^2+12t+16)^3(2t^3+3t^2-3t-5)}{(t^3-3t-1)^9},
\]
\[
F_{9G}(t)=\frac{(t^3-9t-12)(9-3t^3)(5t^3+18t^2+18t+3)}{(t^3+3t^2-3)^3}.
\]
}
\caption{The genus-$0$ groups of $3$-power level from \cite[Table~1]{34}, with the maps between the corresponding modular curves, taken from \cite{SZsupp}.}\label{tab:SZ9}
\end{table}

Finally, combining the actions on $\ell$-power torsion points for a fixed prime $\ell$, one obtains the \emph{$\ell$-adic representation}
\( \rho_{E,\ell^{\infty}}\colon G_{\Q}\longrightarrow \GL_2(\Z_{\ell}), \)
defined as the inverse limit of the representations $\rho_{E,\ell^k}$, $k\geqslant 1$. For a non-CM elliptic curve $E$, its image $H=\rho_{E,\ell^{\infty}}(G_{\Q})$ is an open subgroup of $\GL_2(\Z_{\ell})$, and the \emph{level} of $H$ is the smallest power $\ell^k$ such that $H$ is the full preimage of its reduction modulo $\ell^k$. Rouse, Sutherland and Zureick-Brown \cite{adic} provided a near-complete classification of the possible $\ell$-adic images for elliptic curves over $\Q$, based on the study of rational points on the corresponding modular curves of prime-power level. For $\ell=3$, the only case left open by \cite[Theorem 1.6]{adic} -- that $H$ is contained in the normaliser of the non-split Cartan subgroup of level $27$ -- has recently been ruled out in \cite{BBHJM}. %by Balakrishnan, Betts, Hast, Jha and M\"uller
%The classification of $3$-adic images of elliptic curves over $\Q$ is now complete.%, and we record the consequence that we need.

\begin{corollary}[Consequence of {\protect{\cite[Theorem 1.6]{adic}}} and {\protect{\cite{BBHJM}}}]\label{kor:RSZB3}
Let $E/\Q$ be an elliptic curve without complex multiplication and let $H = \rho_{E,3^\infty}(G_{\Q})$. Then $\langle H,-I\rangle$ is conjugate to one of the $22$ groups of Table~\ref{tab:SZ9}. In particular, $G_E(9)$ is either the reduction modulo~9 of one of these groups or an index two subgroup of it not containing $-I$.
\end{corollary}
\begin{proof}
By \cite[Theorem 1.6]{adic}, exactly one of the following holds: (i) $X_{\langle H, -I\rangle}$ is isomorphic to $\PP^{1}$ or to an elliptic curve of rank one, and $\langle H, -I\rangle$ is listed in \cite[Tables 1--4]{34}; (ii) $X_{\langle H, -I\rangle}$ has an exceptional rational point and $H$ is listed in \cite[Table 1]{adic}; (iii) $H$ is contained in the normaliser of the non-split Cartan subgroup of level $27$. No group of $3$-power level appears in \cite[Table 1]{adic}, case (iii) has been ruled out in \cite{BBHJM}, and all the groups of $3$-power level in \cite[Tables 1--4]{34} have genus $0$ and appear in \cite[Table 1]{34}; they are reproduced in Table~\ref{tab:SZ9}.
\end{proof}

By \cite[Corollary~1.6]{34}, there are exactly $47$ subgroups $H \le \GL_2(\Z_3)$ that arise as $3$-adic images of infinitely many non-CM elliptic curves over $\Q$ (the groups of Table~\ref{tab:SZ9} and certain index~2 subgroups of them). Together with \cite{adic} and \cite{BBHJM}, this completes the list of $3$-adic images of non-CM elliptic curves over $\Q$.

%When $E$ admits a rational $3$-isogeny, the classification takes a particularly concrete form.

%\begin{corollary}[{\protect{\cite[\S 1.3, Corollary 1.1]{adic}}}]\label{kor:RSZBisog}
%Let $E/\Q$ be an elliptic curve without CM admitting a rational $3$-isogeny, and let $\overline{H}=\langle \rho_{E,3^{\infty}}(G_{\Q}),-I\rangle$. Then $[\GL_2(\Z_3):\overline{H}]\in\{4,12,36\}$, and the modular curve $X_{\overline{H}}$ has genus $0$ and infinitely many rational points. In the labels of Table~\ref{tab:SZ9}, $\overline{H}$ is conjugate to one of $3B^{0}-3a$, $3D^{0}-3a$, $9B^{0}-9a$, $9C^{0}-9a$, one of the nine groups $9H^{0}-9a$ through $9J^{0}-9c$, or $27A^{0}-27a$.
%\end{corollary}

\section{Results} %\vskip -1em
Our result is the following theorem, and this section is entirely devoted to proving it.

\begin{theorem}\label{tm:no318}
    Let $E/\mathbb{Q}$ be an elliptic curve without complex multiplication, and let $K$ be a sextic number field. The group $E(K)_{tors}$ does not contain a subgroup isomorphic to $C_{3} \oplus C_{18}$.
\end{theorem}

\begin{lemma}\label{lem:mod3}
Let $E/\Q$ be an elliptic curve without CM, and let $K$ be a number field of degree $6$ such that $E[3]\subseteq E(K)$. Then $G_E(3)\in\{\texttt{3B.1.1},\ \texttt{3B.1.2},\ \texttt{3Cs.1.1}\}$.
\end{lemma}
\begin{proof}
Since $E[3]\subseteq E(K)$, the $3$-division field satisfies $\Q(E[3])\subseteq K$, so $[\Q(E[3]):\Q]=|G_E(3)|$ divides $6$. Consulting Table~\ref{tab:mod3}, we conclude $G_E(3) \in \{\texttt{3Cs.1.1}, \texttt{3B.1.1}, \texttt{3B.1.2}\}$.
\end{proof}

\noindent By \cite[Theorem 18]{tomi1}, it remains to study the case
\[
G_E(2)=\texttt{2B}\quad\text{and}\quad G_E(3)\in\{\texttt{3B.1.1},\ \texttt{3B.1.2},\ \texttt{3Cs.1.1}\}.
\]

Part of the proof of the following proposition already appeared in \cite[Proof of Theorem 18]{tomi1}, in the course of proving that if $E(K)_{\tors}\supseteq C_3\oplus C_{18}$, then $G_E(2)$ must be $\texttt{2B}$. %We use Sutherland labels in this proposition.
%\noindent

\begin{proposition}\label{prop:Borel}
Let $E/\Q$ be an elliptic curve without CM. 
If $G_E(2)=\texttt{2B}$, $G_E(3) \in \{$\texttt{3B.1.1}, \texttt{3B.1.2}$\}$, 
then $E(K)_{\tors}$ cannot contain $C_3 \oplus C_{18}$ for any sextic number field $K$.
\end{proposition}
\begin{proof}
As recalled in Section~\ref{sec:Galois}, $j(E) = j_{\texttt{2B}}(y)=\frac{256(y+1)^3}{y}$ for some $y\in\Q^{\times}$. Since the generators for both \texttt{3B.1.1} and \texttt{3B.1.2} in Table~\ref{tab:mod3} are upper-triangular but not diagonal matrices, they stabilize a unique $1$-dimensional subspace. Hence $E$ admits exactly one rational $3$-isogeny. The existence of full $3$-torsion over $K$ implies that $\Q(E[3])\subseteq K$, and their common degree of $6$ ($|G_E(3)|=6$, see Table \ref{tab:mod3}) forces the two fields to coincide, i.~e., $K=\Q(E[3])$.
 We now divide the proof into two cases.

1.) Assume $E$ does not have a rational $9$-isogeny. We verify that the hypotheses of \cite[Lemma 6.13]{compositum} are satisfied. As noted above, $E$ admits exactly one rational $3$-isogeny.

Since $|G_E(3)|=6$ and the generators displayed in Table~\ref{tab:mod3} do not commute, both groups $\texttt{3B.1.1}$ and $\texttt{3B.1.2}$ are nonabelian of order $6$, hence isomorphic to $S_3$. In particular $\Gal(K/\Q)\cong S_3$, and therefore
\(
K\subseteq \Q(3^\infty)
\)
by \cite[Theorems 3.5 and 3.6]{compositum}, where $\Q(3^\infty)$ denotes the compositum of all cubic extensions of $\Q$. Since
\(
C_3\oplus C_{18}\subseteq E(K)_{\tors},
\)
the field $K$ contains a point of order $9$, and hence so does $\Q(3^\infty)$. Thus all hypotheses of \cite[Lemma 6.13]{compositum} are satisfied, and we obtain
\(
j(E)=\frac{(x+3)(x^2-3x+9)(x^3+3)^3}{x^3}
\)
for some $x\in\Q^\times$.
Hence we must have the relation
\[
\frac{(x+3)(x^2 - 3x + 9)(x^3+3)^3}{x^3} 
= \frac{256(y+1)^3}{y}, \quad \text{ for some } x,y \in \Q^\times.  
\]

\noindent Clearing denominators leads to a curve which is birational to $
X: \,\, y^2 + (x^3+1)y = -9x^3$.
This is a genus $2$ hyperelliptic curve whose Jacobian has rank $0$ over $\Q$.  
A short computation in Magma (see code \href{https://github.com/NikolaAdzaga/SexticTorsion/blob/main/no_rat9isogeny_2B_vs_3Bs.m}{\texttt{no\_rat9isogeny\_2B\_vs\_3Bs.m}}) shows that the rational points on $X$ correspond to CM points (elliptic curves with $j$-invariants $0$ and $54000$, both of which are CM elliptic curves).

2.) If $E$ has a rational $9$-isogeny, let \(
H=\rho_{E,3^\infty}(G_{\Q})
\)
be the $3$-adic image. By Corollary~\ref{kor:RSZB3}, the group $\langle H,-I\rangle$ is one of the groups of Table~\ref{tab:SZ9}.
Let $\overline{H}=\langle H,-I\rangle$ be the corresponding group of Table \ref{tab:SZ9}, and let $\overline{H}_9\leq \GL_2(\Z/9\Z)$ denote its reduction modulo $9$. Since $\overline{H}_9=\langle G_E(9),-I\rangle$, the actual image $G_E(9)$ is either $\overline{H}_9$ itself or an index-$2$ subgroup $H'\leq \overline{H}_9$ such that
\(
\overline{H}_9=\langle H',-I\rangle\).

These groups $\overline{H}$ have level $1$, $3$, $9$ or $27$.
As $E$ has a rational $9$-isogeny, $\overline{H}$ stabilizes a cyclic subgroup of order $9$ in $(\Z/9\Z)^2$. If the level of $\overline{H}$ divided $3$, then $\overline{H}_9$ would contain the whole kernel of $\GL_2(\Z/9\Z)\to \GL_2(\F_3)$, which stabilizes no such subgroup. The only group of level $27$ is $27A^{0}-27a$, and its reduction modulo $9$ is conjugate to $9B^{0}-9a$. Therefore $\overline{H}_9$ is conjugate to one of the $16$ groups of level $9$ in
Table~\ref{tab:SZ9}. For each of them we run through all candidates $H'=G_E(9)$
(denoted \texttt{Hsub} in
\href{https://github.com/NikolaAdzaga/SexticTorsion/blob/main/rat9isogeny.m}{\texttt{rat9isogeny.m}}),
imposing the following necessary conditions.

\begin{itemize}
    \item[(a)] The reduction of $H'$ modulo $3$ must be conjugate to one of the groups $\texttt{3B.1.1}$ or $\texttt{3B.1.2}$.%, because by assumption
    %\( G_E(3)\in \{\texttt{3B.1.1},\texttt{3B.1.2}\}\).
%If $H'=G_E(9)$ is the actual mod $9$ image, then under the natural reduction map
%\( \GL_2(\Z/9\Z)\longrightarrow \GL_2(\F_3) \)
%its image is exactly $G_E(3)$. Therefore the reduction of $H'$ modulo $3$ must be conjugate to one of the groups $\texttt{3B.1.1}$ or $\texttt{3B.1.2}$.
    \item[(b)] A point of order $9$ is defined over $K$ if and only if it is fixed by the image of $\Gal(\overline{\Q}/K)$ under the mod $9$ representation. Since $K=\Q(E[3])$, $\Gal(\overline{\Q}/K)=\ker\rho_{E,3}$, so for \mbox{$H'=G_E(9)$} the subgroup $\ker\left(H'\to \GL_2(\F_3)\right)$ must fix a vector of order $9$ in $(\Z/9\Z)^2$.

    \item[(c)] As $E$ has a rational $9$-isogeny, $H'$ must stabilize a cyclic subgroup of order $9$ in $(\Z/9\Z)^2$.
\end{itemize}

A Magma computation and its \href{https://github.com/NikolaAdzaga/SexticTorsion/blob/main/rat9isogeny.log}{log} show that conditions (a) and (b) leave only two possible labels:
$9I^0-9c$ and $9J^0-9a$.
A further computation shows that for both surviving subgroups $H'$ of the group $\overline{H}_9=9J^0-9a$, condition \emph{(c)} fails: there is no $H'$-stable cyclic subgroup of order $9$ in $(\Z/9\Z)^2$. Hence the label $9J^0-9a$ cannot occur in the rational-$9$-isogeny case.

Therefore the only remaining possibility is the label $9I^0-9c$. We consider the fiber product
\(
9I^{0}-9c \times_{X(1)} \texttt{2B},
\)
given by the equation $j_{9I^0-9c}(t)=j_{\texttt{2B}}(s)$. This is \href{https://beta.lmfdb.org/ModularCurve/Q/18.108.2.4/}{the modular curve listed in the LMFDB under the RSZB label \texttt{18.108.2.4}} (for further details on this labelling system, one can consult \cite[Sec.~2.4]{adic}). It is a curve of genus $2$, isomorphic over $\Q$ to the classical modular curve $X_1(18)$, whose rational points are all cusps by Mazur's theorem \cite{mazurtorzija}. Hence the only corresponding values are $(t,s)=(1,0)$ and $(0,0)$, and since $s=0$ is a pole of $j_{\texttt{2B}}$, no elliptic curve arises from this intersection. This agrees with a \href{https://github.com/NikolaAdzaga/SexticTorsion/blob/main/X9I0c_vs_2B.m}{direct computation in Magma}.
\end{proof}

\begin{remark}
A reader might wonder whether the proof above can be shortened by studying the equation $j_\texttt{2B}(s)=j_{9J^0-9a}(t)$, thereby avoiding the case split. However, this leads to a curve of genus $4$, already mentioned in \cite[p.~343]{tomi1}, and we do not know how to determine all of its rational points.
\end{remark}

We now treat the remaining case $G_E(3)=\texttt{3Cs.1.1}$, following the same strategy as in the second case of the proof of Proposition~\ref{prop:Borel}.

\begin{proposition}\label{prop:Cs}
Let $E/\Q$ be an elliptic curve without CM.
If $G_E(2)=\texttt{2B}$, $G_E(3)=\texttt{3Cs.1.1}$,
then $E(K)_{\tors}$ cannot contain $C_3 \oplus C_{18}$ for any sextic number field $K$.
\end{proposition}

\begin{proof}
Assume that $C_3\oplus C_{18}\subseteq E(K)_{\tors}$, and let $P_9\in E(K)$ be a point of order $9$. Since $E[3]\subseteq E(K)$, we have $\Q(E[3])\subseteq K$, and $[\Q(E[3]):\Q]=|G_E(3)|=2$ (Table~\ref{tab:mod3}), so $[K:\Q(E[3])]=3$.

Let $H=\rho_{E,3^{\infty}}(G_{\Q})$ be the $3$-adic image. By Corollary~\ref{kor:RSZB3}, the group $\overline{H}=\langle H,-I\rangle$ is one of the groups of Table~\ref{tab:SZ9}, and, as in the proof of Proposition~\ref{prop:Borel}, the actual image $G_E(9)$ is either $\overline{H}_9$, the reduction of $\overline{H}$ modulo $9$, or an index-$2$ subgroup $H'\leq \overline{H}_9$ such that $\langle H',-I\rangle=\overline{H}_9$. Since no rational $9$-isogeny is available here, we run through the reductions modulo $9$ of \emph{all} the groups of Table~\ref{tab:SZ9} and impose on the candidates $H'=G_E(9)$ analogous conditions (a) and (b):

\begin{itemize}
    \item[(a)] The reduction of $H'$ modulo $3$ must be conjugate to $\texttt{3Cs.1.1}$.
    \item[(b)] A point of order $9$ is defined over $K$ if and only if it is fixed by $I=\rho_9(\Gal(\overline{\Q}/K))$. The fixed field of $I$ inside $\Q(E[9])$ is $K\cap\Q(E[9])$. Since $\Q(E[3]) \subseteq K$, we have
    \(
    I \leq N=\ker\left(H'\to \GL_2(\F_3)\right)
    \)
    and $[N:I]$ divides $[K:\Q(E[3])]=3$. Since $I$ fixes the vector $v$ corresponding to $P_9$, the orbit of $v$ under $N$ must have size dividing $3$. (Equivalently, $[\Q(E[3],P_9):\Q]$, which is twice the size of this orbit, must divide $[K:\Q]=6$.)
\end{itemize}

The Magma computation in \href{https://github.com/NikolaAdzaga/SexticTorsion/blob/main/3Cs11_subgroup_search.m}{\texttt{3Cs11\_subgroup\_search.m}} (with its output in the accompanying \href{https://github.com/NikolaAdzaga/SexticTorsion/blob/main/3Cs11_subgroup_search.log}{log}) shows that for every candidate except two, either the reduction modulo $3$ fails to be conjugate to $\texttt{3Cs.1.1}$, or every orbit of $N$ on the vectors of order $9$ in $(\Z/9\Z)^2$ has size $9$. For the two surviving candidates,
$\overline{H}_9 = 9H^{0}-9b=\big\langle
\left(\begin{smallmatrix}
1 & 0 \\
3 & 1
\end{smallmatrix}\right),
\left(\begin{smallmatrix}
5 & 3 \\
0 & 2
\end{smallmatrix}\right),
\left(\begin{smallmatrix}
2 & 0 \\
1 & 1
\end{smallmatrix}\right)
\big\rangle$, whose $j$-map is
 \small \[j_{9H^{0}-9b}(t)=\frac{(t^3-3t^2-9t+3)^3(t^3+9t^2-9t-9)^3(t^6-18t^5+171t^4+180t^3-297t^2-162t+189)^3}{8(t^2-1)^3(t^2+3)^9(t^3-9t^2-9t+9)^3}. \]
\normalsize
\noindent It remains to rule out this group. As in the proof of Proposition~\ref{prop:Borel}, the fiber product $9H^{0}-9b \times_{X(1)} \texttt{2B}$, given by $j_{9H^{0}-9b}(t)=j_{\texttt{2B}}(s)$, is the curve with RSZB label \texttt{18.108.2.1}, which is again isomorphic over $\Q$ to $X_1(18)$. Hence all its rational points are cusps, so every affine one has $s=0$, a pole of $j_{\texttt{2B}}$, and no elliptic curve arises.
\end{proof}

\begin{proof}[Proof of Theorem \ref{tm:no318}]    
Assume that $C_{3} \oplus C_{18} \subseteq E(K)_{tors}$ for $E/\Q$ and a sextic field $K$. As recalled in Section~\ref{sec:Galois}, $G_{E}(2)=\texttt{2B}$, and since $K$ contains the full $3$-torsion, Lemma~\ref{lem:mod3} gives $G_E(3)\in \{ \texttt{3Cs.1.1},\, \texttt{3B.1.1},\, \texttt{3B.1.2}\}$. Proposition \ref{prop:Borel} excludes the last two (Borel) groups, and Proposition \ref{prop:Cs} excludes $\texttt{3Cs.1.1}$. We conclude that there does not exist an elliptic curve $E/\mathbb{Q}$ without CM and a sextic number field $K$ such that $C_{3} \oplus C_{18} \subseteq E(K)_{tors}$.
\end{proof}

Finally, since $C_3\oplus C_{18}$ cannot occur as the torsion subgroup of CM curves $E/K$ with $K$ sextic \cite[§4.6]{CMCCRS}, and together with \cite[Theorem 1]{tomi1}, we conclude Theorem \ref{tm:complete} also holds.

\subsection*{Declaration of generative AI and AI-assisted technologies in the manuscript preparation process}
\noindent During the preparation of this work, the authors used ChatGPT (OpenAI), Claude (Anthropic) and Gemini (Google) to assist with reorganizing portions of the exposition, and with expanding background explanations, including material related to subgroups of $\mathrm{GL}_2$, guided by \cite[Chap.~XI]{LangMF}. The tools were also used to help refactor and document the accompanying code, and to improve overall readability. The first version of the manuscript and code was written without AI tools. After using these tools, the authors carefully reviewed, checked, and edited all content and code as needed and take full responsibility for the content of this manuscript.

%\vskip -3em
%For elliptic curves $E$ and a given sextic field $K$, it is an open problem to determine the possible torsion structures $E(K)_{\tors}$
\subsection*{Acknowledgements}
\noindent We are grateful to an anonymous referee for their thorough review and insightful comments. In particular, we thank them for identifying an error in the previous proof of Proposition \ref{prop:Borel}, and for showing us how to simplify the arguments and obtain stronger results elsewhere in the earlier version of the paper. We thank another anonymous referee for suggesting a significant simplification of our proof. We thank Filip Najman for careful reading and numerous comments that improved this manuscript. We thank Maarten Derickx for remarks on the geometric properties of Galois representations that influenced the revision of Appendix~\ref{app:background}.
We also thank an anonymous referee of an earlier version of this manuscript for their helpful comments. %N.\ A.\ was supported by the institutional project "Graphs, elliptic and modular curves" funded by the University of Zagreb Faculty of Civil Engineering. This work was supported by the Croatian Science Foundation under the project number HRZZ IP-2022-10-5008, and by

This study was conducted within the framework of the institutional project of the University of Zagreb Faculty of Civil Engineering – [GEM
– Graphs, elliptic and modular curves], under the programme line for enhancing scientific excellence, financed through
the National Recovery and Resilience Plan 2021–2026. The project is funded by the European Union – NextGenerationEU. This work was supported by the Croatian Science Foundation under the project number HRZZ IP-2022-10-5008, and by the project “Implementation of cutting-edge research and its application as part of the Scientific Center of Excellence for Quantum and Complex Systems, and Representations of Lie Algebras'', Grant No.~PK.1.1.10.0004, co-financed by the European Union through the European Regional Development Fund - Competitiveness and Cohesion Programme 2021-2027.
%\nocite{*}
\printbibliography[heading=nomark]

\newpage
\appendix
\section{A dictionary of labels}\label{app:labels}

Throughout the paper, each group is labelled as in the source cited at that point, so that our statements can be compared with that source directly. Several labelling systems therefore occur, and we collect here a dictionary translating between them.
{\small{
\begin{table}[H]
\centering
\begin{tabular}{@{}lccc@{}}
Label used here & LMFDB label & RSZB label & Corresponding modular curve \\
\addlinespace
\texttt{2B} & 2.3.0.a.1 & 2.3.0.1 & $X_0(2)$ \\
\texttt{3B.1.1} & 3.8.0-3.a.1.2 & 3.8.0.1 & $X_1(3)$ \\
\texttt{3B.1.2} & 3.8.0-3.a.1.1 & 3.8.0.2 &  \\
\texttt{3Cs.1.1} & 3.24.0-3.a.1.1 & 3.24.0.1 & $X_{\mathrm{arith}}(3)$ \\
\addlinespace
$3A^{0}-3a$ & 3.3.0.a.1 & 3.3.0.1 & $X_{\mathrm{ns}}^{+}(3)$ \\
$3B^{0}-3a$ & 3.4.0.a.1 & 3.4.0.1 & $X_0(3)$ \\
$3C^{0}-3a$ & 3.6.0.b.1 & 3.6.0.1 & $X_{\mathrm{sp}}^{+}(3)$ \\
$3D^{0}-3a$ & 3.12.0.a.1 & 3.12.0.1 & $X_{\mathrm{sp}}(3)$ \\
\addlinespace
$9A^{0}-9a$ & 9.9.0.a.1 & 9.9.0.1 &  \\
$9B^{0}-9a$ & 9.12.0.a.1 & 9.12.0.1 & $X_0(9)$ \\
$9C^{0}-9a$ & 9.12.0.b.1 & 9.12.0.2 &  \\
$9D^{0}-9a$ & 9.18.0.d.1 & 9.18.0.2 &  \\
$9E^{0}-9a$ & 9.18.0.a.1 & 9.18.0.1 &  \\
$9F^{0}-9a$ & 9.27.0.a.1 & 9.27.0.1 &  \\
$9G^{0}-9a$ & 9.27.0.b.1 & 9.27.0.2 &  \\
$9H^{0}-9a$ & 9.36.0.b.1 & 9.36.0.2 &  \\
$9H^{0}-9b$ & 9.36.0.a.1 & 9.36.0.1 &  $X_{9H^{0}-9b} \cong_{\Q} X_1(9)^{\dagger}$ \\
$9H^{0}-9c$ & 9.36.0.c.1 & 9.36.0.3 &  \\
$9I^{0}-9a$ & 9.36.0.d.2 & 9.36.0.5 &  \\
$9I^{0}-9b$ & 9.36.0.e.1 & 9.36.0.6 &  \\
$9I^{0}-9c$ & 9.36.0.d.1 & 9.36.0.4 & $X_{1}(9)$ \\
$9J^{0}-9a$ & 9.36.0.f.2 & 9.36.0.7 &  \\
$9J^{0}-9b$ & 9.36.0.f.1 & 9.36.0.9 &  \\
$9J^{0}-9c$ & 9.36.0.g.1 & 9.36.0.8 &  \\
\addlinespace
$27A^{0}-27a$ & 27.36.0.a.1 & 27.36.0.1 &  \\
\hdashline
\addlinespace
18.108.2.1 & 18.108.2.a.1 & 18.108.2.1 & $X_1(18)\cong_{\Q} \, 9H^{0}-9b \times_{X(1)} \texttt{2B}$ \\
18.108.2.4 & 18.108.2.d.1 & 18.108.2.4 & $X_1(18) \cong_{\Q} \,9I^{0}-9c \times_{X(1)} \texttt{2B}$ \\
\end{tabular}
\caption{Labels of the subgroups and modular curves occurring in this paper. $^{\dagger}$Isomorphic as curves over $\Q$, but with a different $j$-map.}\label{tab:labels}
\end{table}}}

\newpage
\begin{refsection}
\section{Background on Galois representations and Serre's uniformity question}\label{app:background}
\noindent This appendix collects expository material that is not needed for the proofs, but may help a reader who is not acquainted with Galois representations attached to elliptic curves.

Let us see explicitly how the Galois representation is obtained in the case where \(n = p\) is prime. Let
\(\sigma \in \mathrm{Gal}(\mathbb{Q}(E[p])/\mathbb{Q})\). Then \(E[p]\) is
generated by two points of order \(p\), say \(P_1\) and \(P_2\). Thus, for
some \(P_1^\sigma \in E[p]\), we have \(
P_1^\sigma = \alpha P_1 + \beta P_2, \text{with } \alpha \neq 0 \text{ or } \beta \neq 0,\) and, similarly, \(
P_2^\sigma = \gamma P_1 + \delta P_2, \text{with } \gamma \neq 0 \text{ or } \delta \neq 0.\)
Therefore,
\(
\rho_p(\sigma) =
\begin{pmatrix}
\alpha & \gamma \\
\beta  & \delta
\end{pmatrix}
\in \mathrm{GL}_2(\mathbb{F}_p).
\) The fact that \(\rho_p(\sigma) \in \mathrm{GL}_2(\mathbb{F}_p)\), i.e. that
we obtain an invertible matrix, follows from the fact that \(\sigma\) is an
automorphism, and hence has an inverse. In the same way, when \(n\) is an arbitrary integer (not necessarily prime), we obtain that \(\rho_n\) is a homomorphism \( \rho_n \colon G_{\Q} \longrightarrow \mathrm{GL}_2(\mathbb{Z}/n\mathbb{Z}), \)
where $G_{\Q}=\mathrm{Gal}(\overline{\mathbb{Q}}/\mathbb{Q})$ is the absolute Galois group, as in Section~\ref{sec:Galois}.
We continue with \href{https://github.com/NikolaAdzaga/SexticTorsion/blob/main/Galois_image_rho3.m}{an example}.

\noindent\textbf{Example.}  
Consider the elliptic curve $
E/\Q \colon y^2 + y = x^3 - x$.
The $\Qbar$-points of order \(3\) on \(E\) form the group \(E[3] \cong \mathbb{Z}/3\mathbb{Z} \oplus \mathbb{Z}/3\mathbb{Z}\).
The eight non-trivial points of order $3$ are obtained by performing the group law twice to get the general coordinates of $3P$ and solving $3P = \mathcal{O}_E$. The coordinates of points in $E[3]$ generate a number field
\(\mathbb{Q}(E[3])\) of degree \( 48 \) over \(\mathbb{Q}\). The associated Galois representation
\(
\rho_3 \colon G_{\Q} % \mathrm{Gal}(\Qbar/\mathbb{Q})
\to
\mathrm{GL}_2(\mathbb{Z}/3\mathbb{Z})
\)
describes how the absolute Galois group permutes the \(3\)-torsion points of
\(E\). In this case, the image of \(\rho_3\) is the full group
\(\mathrm{GL}_2(\mathbb{F}_3)\), which has order \(48\).
Equivalently, we can say that
$G_E(3) = \Image \rho_3 \cong \mathrm{Gal}(\mathbb{Q}(E[3])/\mathbb{Q}) \cong \mathrm{GL}_2(\mathbb{F}_3)$.
\href{https://github.com/NikolaAdzaga/SexticTorsion/blob/main/Galois_image_rho3.m}{The code} for computing this in Magma is provided at GitHub accompanying the paper, but one can also consult \href{https://www.lmfdb.org/EllipticCurve/Q/37/a/1}{this elliptic curve's LMFDB page}, which states that $\rho_3$ has maximal image.

It turns out that the situation of obtaining the whole general linear group is typical. Recall that Serre's open image theorem implies that for each non-CM elliptic curve, there are only finitely many primes $p$ such that $\rho_p$ is not surjective. Moreover, Serre posed the following question. Is there a uniform bound $C$ such that for all primes $p>C$ and all elliptic curves $E$ without CM, the mod $p$ Galois representation $\rho_p$ is surjective? This question has motivated a lot of research and it is widely believed that the answer to his question is positive.

%For more exposition on $\ell$-adic representations and Serre's open image theorem, one can consult \cite[Section 1]{BLV09}.
The recent preprint \cite{Kenji} approaches the study of torsion over number fields from a perspective similar to ours. That paper is notable for both its general results and its accessible introduction and application of \emph{adelic representations}.

Not every subgroup of $\GL_2(\mathbb{Z}/n\mathbb{Z})$ can be the image of $\rho_{E,n}$. To address this, Zywina introduced the following definition. We say that a subgroup 
\( G \subseteq \GL_2(\mathbb{Z}/n\mathbb{Z}) \) is \emph{applicable} if
\begin{itemize}
    \item it is proper (\( G \neq \GL_2(\mathbb{Z}/n\mathbb{Z}) \));
    \item it contains $-I$ and $\det(G) = (\mathbb{Z}/n\mathbb{Z})^\times$;
    \item it has an element of trace $0$ and determinant $-1$ fixing a point of order $n$ in $(\mathbb{Z}/n\mathbb{Z})^2$.
\end{itemize}
\noindent As Zywina showed, if $E/\Q$ is an elliptic curve such that $\rho_{E,n}$ is not surjective, then the group $\pm \Image\rho_{E,n} = \langle \Image\rho_{E,n}, -I\rangle$ is an applicable subgroup of $\GL_2(\mathbb{Z}/n\mathbb{Z})$ \cite[Prop.~2.2]{zywina}.

We now recall the maximal subgroups of $\GL_2(\Fp)$.

For a prime number $p$, let $B(p)$ %\texttt{pB}$
denote the \emph{standard Borel subgroup} of $\GL_2(\Fp)$, i.~e.~the subgroup of invertible upper triangular matrices. A \emph{Borel subgroup} is any conjugate of the standard Borel subgroup, and sometimes subgroups of $\GL_2(\Fp)$ are identified up to conjugation (equivalently, up to change of basis of $\Fp^2$).
 This subgroup, as well as the others introduced below, plays an important role in understanding the subgroup structure of $\GL_2(\Fp)$. Concretely, if $G$ is a subgroup of $\GL_2(\Fp)$ with order divisible by $p$, then $G$ is contained in a Borel subgroup (equivalently, stabilizes a line in $\Fp^2$), or $\SL_2(\Fp)\subseteq G$.

Define the following matrices in $\mathrm{GL}_2(\Fp)$, for $a, b \in \Fp$ and a quadratic non-residue $\phi\in \Fp$:
\vskip -0.5em
\[
D(a,b) = 
\begin{bmatrix}
a & 0 \\
0 & b
\end{bmatrix}, \quad
M_\phi(a,b) = 
\begin{bmatrix}
a & b\phi \\
b & a
\end{bmatrix}, \quad
T = 
\begin{bmatrix}
0 & 1 \\
1 & 0
\end{bmatrix}, \quad
J = 
\begin{bmatrix}
1 & 0 \\
0 & -1
\end{bmatrix}.
\]

We now introduce some notation, namely the following subgroups of $\GL_{2}(\Fp)$. \emph{The split Cartan subgroup} consists of diagonal matrices \[ \pCs= \{ D(a,b) : a,b \in (\Fp)^{\times} \}.\]
This subgroup is isomorphic to $\Fp^\times \times \Fp^\times$ and has index $2$ in its normalizer, which consists of diagonal and anti-diagonal matrices: 
\[ \normSplitCartan{p} = \pCs\cupw T\cdot\pCs = \{ D(a,b), \; T \cdot D(a,b) : a,b \in (\Fp)^{\times} \}.\]
Geometrically, the matrix $T$ acts on the split Cartan by permuting the diagonal entries, which corresponds to swapping the two eigenspaces (the automorphism $(x,y) \mapsto (y,x)$ of $\Fp \times \Fp$).

\emph{The non-split Cartan subgroup} is
\[ \pNs = \{ M_{\phi}(a,b) : (a,b) \in (\Fp)^{2}, (a,b) \neq (0,0) \}.\]
This subgroup is isomorphic to $\F_{p^2}^\times$. Its normalizer is denoted by
\[ \normNonSplit{p}= \pNs \cupw J\cdot\pNs = \{ M_{\phi}(a,b), \; J \cdot M_{\phi}(a,b) : (a,b) \in (\Fp)^{2}, (a,b) \neq (0,0) \}.\]
Here, the action of $J$ corresponds to the Frobenius automorphism (Galois conjugation) on the underlying quadratic extension $\F_{p^2}$.

Cartan subgroups $\pCs$ and $\pNs$ are the maximal diagonalizable
parts of $\GL_2(\Fp)$, the distinction being whether their eigenvalues
lie in $\F_p$ (split case) or its quadratic extension (non-split case). Cartan subgroups can likewise be defined over more general fields of
characteristic~$p$, and the following statement holds for finite
subgroups even in that setting. A subgroup $G$ of $\GL_2(\F_p)$ of order
prime to $p$ is either contained in the normalizer of a Cartan
subgroup, or $G/\F_p^{\times}I$, the projective image of $G$, is
isomorphic to $A_4$, $S_4$, or $A_5$ (the exceptional permutation groups).

The description of subgroups of $\GL_2(\Fp)$ is usually referred to as Dickson's classification of maximal subgroups \cite{Dickson}. %Before turning to Galois representations,
For proofs and more details on the classification of subgroups of $\GL_2$ over fields of positive characteristic, see \cite[Chap.~XI]%., \emph{where the going is comparatively easy}]
{LangMF}.

The classification of maximal subgroups of $\GL_2(\Fp)$ plays a central role in the general approach to Serre's question. 
The strategy is to show that, for $p$ sufficiently large, there are no elliptic curves without complex multiplication for which the image of $\rho_{E,p}$ is contained in one of these maximal subgroups. 
The exceptional cases were solved by Serre in \cite{Serre}. 
In \cite{mazur:isogenies}, Mazur handled the Borel case by determining all possible prime degrees of rational isogenies of elliptic curves over $\Q$. 
Bilu, Parent, and Rebolledo studied the case of the normaliser of a split Cartan. 
In \cite{BPR}, they proved that if $E/\Q$ is an elliptic curve without CM, and $p \geq 11$ is a prime different from $13$, then the image of $\rho_{E,p}$ is not contained in $\normSplitCartan{p}$. 
The remaining case $p=13$ was settled in \cite{BDMTV}.
To fully answer the question of Serre, it remains to rule out the possibility that the image of the mod $p$ Galois representation of any non-CM elliptic curve defined over $\Q$ is contained in the normaliser of the non-split Cartan subgroup of $\GL_2(\Fp)$. 
This is still an open problem, but important progress has been made. 
Le Fourn and Lemos \cite{LeFournLemos2021} proved that if $p > 1.4 \cdot 10^7$ and $E/\Q$ is an elliptic curve without CM, then the image of $\rho_{E,p}$ is either $\GL_2(\Fp)$ or the full normaliser of a non-split Cartan subgroup of $\GL_2(\Fp)$. 
More recently, Furio and Lombardo \cite[Theorems 1.5 and 1.6]{LombardoFurio} (together with \cite[Proposition 1.13]{zywina}) proved that for every prime $p>5$, the image of $\rho_{E,p}$ cannot be a proper subgroup of the normaliser of a non-split Cartan subgroup.

\printbibliography[heading=nomarksub,title={References for this appendix}]
\end{refsection}

\end{document}